\documentclass[12pt]{amsart}
\usepackage{amsmath,amssymb,amsthm,amscd,amsfonts}
\usepackage{verbatim}
\usepackage{comment}
\usepackage{multirow}
\usepackage{mathtools}
\usepackage{enumitem}
\usepackage{pgf,tikz}
\usetikzlibrary{arrows}
\usetikzlibrary{cd} 
\usepackage[normalem]{ulem}
\usepackage{graphicx}
\usepackage{wasysym}
\usepackage{csquotes}
\usepackage{bbm}
\usepackage{bm}
\usepackage{color}
\usepackage{stmaryrd}
\usepackage[margin=1in]{geometry} 

\usepackage{pdfpages}

\newcommand{\Z}{\mathbb{Z}}
\newcommand{\Pj}{\mathbb{P}}
\newcommand{\E}{\mathcal{E}}
\newcommand{\loc}{\mathcal{L}}
\newcommand{\En}{\mathcal{E}_{norm}}

\DeclareMathOperator{\cod}{codim}
\DeclareMathOperator{\cok}{coker}
\DeclareMathOperator{\exc}{expcodim}
\DeclareMathOperator{\socle}{Soc}

\DeclarePairedDelimiter\floor{\lfloor}{\rfloor}

\theoremstyle{theorem}
\newtheorem{teo}{Theorem}
\newtheorem{theoremnonumber}{Theorem}

\newtheorem{propositionnonumber}{Proposition}

\newtheorem{prop}[teo]{Proposition}

\theoremstyle{definition}
\newtheorem{df}[teo]{Definition}
\theoremstyle{remark}
\newtheorem{rmk}[teo]{Remark}

\title{The non-Lefschetz locus of vector bundles of rank $2$ over $\Pj^2$}
\author{Emanuela Marangone}
\email{emarango@nd.edu}
\address{Department of Mathematics \\
	University of Notre Dame \\
	Notre Dame, IN 46556 USA}
\thanks{I would like to express my gratitude and thanks to my advisor, professor Juan Migliore for his guidance and help throughout this project.}

\begin{document}
	\maketitle
	\begin{abstract}
		A finite length graded $R$-module $M$ has the Weak Lefschetz Property if there is a linear element $\ell$ in $R$ such that the multiplication map $\times\ell: M_t\to M_{t+1}$ has maximal rank for every integer $t$. The set of linear forms with this property form a Zariski-open set and its complement is called the non-Lefschetz locus.
		In this paper we study the non-Lefschetz locus for the first cohomology module $H_*^1(\Pj^2,\E)$ of a vector bundle $\E$ of rank $2$ over $\Pj^2$. The main result is to show that this non-Lefschetz locus has the expected codimension under the assumption that $\E$ is general.
	\end{abstract}
\section{Introduction}
 Let $k$ be an algebraically closed field of characteristic zero and let $R=k[x_1,x_2,x_3]$ be the polynomial ring in $3$ variables over $k$. We say that a finite length graded $R$-module $M$ has the Weak Lefschetz Property (WLP) if there is a linear form $\ell$ in $R$ such that the multiplication map $\times\ell: M_t\to M_{t+1}$ has maximal rank for every $t$. In such case $\ell$ is called a Weak Lefschetz element.
 The set of those linear forms is a Zariski-open set and its complement is called the non-Lefschetz locus of $M$.
 
Determining which Artinian graded algebras, or more generally finite length modules, have the Weak Lefschetz Property is a difficult problem and many authors have studied this topic applying tools from algebraic geometry, commutative algebra and combinatorics.
One of the most famous results in this topic was proved separately by Stanley in \cite{28} using algebraic topology,
Watanabe in \cite{29} using representation theory and Reid, Roberts, Roitman  in \cite{26}, and it states that every Artinian monomial complete intersection over a field of characteristic zero has the WLP.
As a consequence, every general complete intersection over a field of characteristic $0$ has the WLP.
However, it is an open question to determine whether every complete intersection has the Weak Lefschetz property. Harima, Migliore, Nagel, and Watanabe in \cite{13} proved 
the result in codimension $3$  introducing the use of the syzygy bundle, that in this case is a free sheaf of rank $2$ over $\Pj^2$, and applying the Grauert-M\"ulich theorem.
 
There are two main papers that inspired this work. The first  \cite{FFP21} by Failla, Flores and Peterson generalizes the case of complete intersections of codimension $3$, proving that for any vector bundle $\E$ of rank $2$ over $\Pj^2$ the first cohomology module $H_*^1(\Pj^2,\E)=\bigoplus_{t\in\Z}H^1(\Pj^2, \E(t))$ has the WLP.
 
The second \cite{main} by Boij, Migliore, Miró-Roig, and Nagel  conjectures that the non-Lefschetz locus of a general complete intersection has the expected codimension (see Remarks \ref{deg}). They proved this conjecture for general complete intersections of codimension $3$ (Theorem 5.3 of \cite{main}) and codimension $4$ (Theorem 5.6 of \cite{main} and see also \cite{io}).
This is particularly important because if the non-Lefschetz locus achieves the expected codimension, its degree is also known \cite{MIGLIORE}. 

The natural question that arises from these two papers is whether, given a vector bundle $\E$ of rank $2$ over $\Pj^2$, the non-Lefschetz locus of $H_*^1(\Pj^2,\E)$ has the expected codimension. The simplest case is when $\E$ is the syzygy bundle of a complete intersection. However even in this case it was shown in \cite{main} that ``most of the time'' a monomial complete intersection has non-Lefschetz locus that is not of the expected codimension, even though it has the WLP.
In this paper we give a complete answer to the question of whether the non-Lefschetz locus has the expected codimension under the assumption that $\E$ is general. 
 \begin{theoremnonumber}[\bfseries \ref{main}]
	The non-Lefschetz locus of $M=H_*^1(\Pj^2, \E)$ has the expected codimension for any general rank $2$ vector bundle $\E$ on $\Pj^2$.
	In particular it is a hypersurface if either $\E$ unstable or $d$ is even and it is a finite set of points if
	$\E$ is stable and $d$ is odd.
\end{theoremnonumber}

This paper is organized as follows.
In section 2 we define the finite length module $M={H}_*^1(\Pj^2, \E)$ and we recall some of the properties of such modules. We include also the definitions of non-Lefschetz locus of $M$ and expected codimension.

In Section 3 we  making more explicit the connection between splitting type and Weak Lefschetz elements in this particular case, following the work of Brenner and Kaid \cite{3} for Artinian algebras with stable syzygy bundle. 
We give a description of the non-Lefschetz locus for the first cohomology module $M={H}_*^1(\Pj^2, \E)$ of any rank $2$  vector bundle $\E$:
\begin{theoremnonumber}[\bfseries \ref{jum}]
	The non-Lefschetz locus is the set of jumping lines of $\E$.
\end{theoremnonumber}

Section 4 gives an positive answer for our question in an easier special case. Here $M$ is a level module and symmetrically Gorenstein and we already know that the non-Lefschetz locus is concentrated in the middle degree, thanks to Flores \cite{F19}. So, in Proposition \ref{b} we prove that in this case, for $\E$ general, the non-Lefschetz locus has the expected codimension.

In section 5, we generalize the results in section 4 to any general vector bundle $\E$ of rank $2$ over $\Pj^2$. We first extend Corollary 5.12 of \cite{F19}.
\begin{propositionnonumber}[\bfseries \ref{new}]
	The non-Lefschetz locus of $M$ is the same as the non-Lefschetz locus in the middle degree:
	$$\loc_M =\loc_{\lfloor \frac{d-4}{2}\rfloor,M}.$$
\end{propositionnonumber} 
After that, we finally prove that the non-Lefschetz locus for the first cohomology module $M={H}_*^1(\Pj^2, \E)$ of any rank $2$  vector bundle $\E$  has the expected codimension (Theorem \ref{main}).	

\section{Preliminaries}
Let $k$ be an algebraically closed field of characteristic zero and let $R=k[x_1,x_2,x_3]$. Fix integers $n\geq 1$, $a_1\leq\dots\leq a_{n+2}$ and $b_1\leq\dots\leq b_n$.
Consider a degree zero graded map $\varphi:\bigoplus_{i=1}^{n+2}R(-a_i)\to\bigoplus_{i=1}^{n} R(-b_i)$ such that $M=\cok(\varphi)$ has finite length.
We have the exact sequence
$$0\to E \to \bigoplus_{i=1}^{n+2}R(-a_i)\to\bigoplus_{i=1}^{n} R(-b_i) \to M \to 0.$$
Since $M$ has finite length, the Buchsbaum-Riemann complex gives us a minimal free resolution of $M$
$$0\to \bigoplus_{i=1}^{n} R(-d+b_i) \to\bigoplus_{i=1}^{n+2}R(-d+a_i)\to \bigoplus_{i=1}^{n+2}R(-a_i)\to\bigoplus_{i=1}^{n} R(-b_i) \to M \to 0$$
where $d=a_1+\dots+a_{n+2}-b_1-\dots -b_n$. 
Sheafifying and decomposing, we get short exact sequences of sheaves
\begin{equation}\label{1}
	0\to \mathcal{E} \to \bigoplus_{i=1}^{n+2}\mathcal{O}_{\mathbb{P}^2}(-a_i)\to\bigoplus_{i=1}^{n} \mathcal{O}_{\mathbb{P}^2}(-b_i) \to 0,
\end{equation}
\begin{equation}\label{2}
	0\to \bigoplus_{i=1}^{n}\mathcal{O}_{\mathbb{P}^2}(-d+b_i)\to \bigoplus_{i=1}^{n+2} \mathcal{O}_{\mathbb{P}^2}(-d+a_i) \to\mathcal{E} \to 0.
\end{equation}
$\mathcal{E}$ is a rank 2 vector bundle on $\Pj^2$ and it satisfies $H^1(\Pj^2, \E(t))=M_t$ so $${H}_*^1(\Pj^2, \E):=\bigoplus_{t\in\Z}H^1(\Pj^2, \E(t))=M.$$
\begin{rmk}
	In general every finite length module $M$ over $R$ that can be expressed as a cokernel of a map $\varphi:\bigoplus_{i=1}^{n+2}R(-a_n)\to\bigoplus_{i=1}^{n} R(-b_i)$ corresponds to a finite length module of the form ${H}_*^1(\Pj^2, \E)$ for a rank $2$ vector bundle  $\E$ on $\Pj^2$, and vice versa. 
\end{rmk}
\begin{df} \label{WLP}
	We say that a finite length graded $R$-module $M$ has the \emph{Weak Lefschetz Property} if there is a linear element $\ell$ in $R$ such that the multiplication map $\times\ell: M_t\to M_{t+1}$ has maximal rank, i.e. is either injective or surjective, for every $t$. In such case $\ell$ is called a \emph{Weak Lefschetz element} or just \emph{Lefschetz element}.
\end{df}
In \cite{FFP21}, the authors proved that the finite length modules  $M=H_*^1(\Pj^2, \E)$ have the Weak Lefschetz Property.

Moreover from \cite{FFP21,F19}:
\begin{itemize}
	\item the Hilbert function of $M$ is unimodal;
	\item the socle degree of $M$ is $e=d-3-b_1$;
	\item the Hilbert function of $M$ is symmetric with respect to the middle degree $\lfloor\frac{e+b_1}{2}\rfloor=\lfloor\frac{d-3}{2}\rfloor$.	
\end{itemize}

Finally, we recall the definition of the non-Lefschetz locus for finite modules over $R$ following \cite{main,MIGLIORE,F19} (here we work in $R=k[x_1,x_2,x_3]$ so we are giving the definitions in only $3$ variables).
\begin{df}
	The \emph{non-Lefschetz locus} of $M$ as subset of $(\Pj^{2})^*$ is
	$$\mathcal{L}_M=\{[\ell]\in\Pj( R_1) : \ell \text{ is not a Lefschetz element of } M \}.$$
\end{df}
\begin{df}
	For any integer $t$ we can also define the \emph{non-Lefschetz locus in degree $t$} 
	$$\mathcal{L}_{M,t}=\{[\ell]\in\Pj(R_1) :\times\ell: M_t\to M_{t+1} \text{ does not have maximum rank} \}.$$
\end{df}

Let $\ell=a_1x_1+a_2x_2+a_3x_3\in R_1$ be a linear form. We introduce $S=k[a_1,a_2,a_3]$ as the homogeneous coordinate ring of the dual projective space $(\Pj^{2})^*$ in the  dual variables $a_1,a_2,a_3$.
Given a choice of basis for $M_t$ and $M_{t+1}$ the map 
\begin{align*}
	 R_1 &\to \hom_k(M_t,M_{t+1})\\
	\ell&\mapsto (\times\ell : M_t \to M_{t+1})
\end{align*}
is represented by a $h_{t+1}\times h_{t}$ matrix $B_t$ of linear forms in $S$, where $h$ is the Hilbert function of $M$.
The locus $\mathcal{L}_{M,t}\subseteq (\Pj^{2})^*$ is scheme-theoretically defined by the ideal $I(\mathcal{L}_{M,t})$ of maximal minors of the matrix $B_t$.
The non-Lefschetz locus $\mathcal{L}_M$ is defined as a sub-scheme of $(\Pj^{2})^*$ by the homogeneous ideal $I(\mathcal{L}_M)=\bigcap_{i\geq0}I(\mathcal{L}_{M,t})$.
\begin{rmk} \label{deg}
		\begin{itemize}
		\item Assume $h_t=\dim M_t\leq \dim M_{t+1}=h_{t+1}$. The expected codimension of $\mathcal{L}_{M,t}$ is $h_{t+1}-h_t+1$, and if it is achieved, then $\deg \mathcal{L}_{M,t}=\binom{h_{t+1}}{h_t-1}$ \cite{main}.
		\item If we consider a Gorenstein algebra $A$ of socle degree $e$, $\mathcal{L}_A=\mathcal{L}_{A,\lfloor\frac{e-1}{2}\rfloor}$ \cite{main}. So under this hypothesis, we can consider the expected codimension and expected degree of the non-Lefschetz locus:
		$$\cod \mathcal{L}_A=h_{\lfloor\frac{e+1}{2}\rfloor}-h_{\lfloor\frac{e-1}{2}\rfloor}+1; \ \ \ \deg \mathcal{L}_A=\binom{h_{\lfloor\frac{e+1}{2}\rfloor}}{h_{\lfloor\frac{e-1}{2}\rfloor}-1}.$$
		We will see later that something similar happens in our case for $M$.
	\end{itemize}
	
\end{rmk}
Here in fact we want to study the codimension of the non-Lefschetz locus of $M$.

\section{Jumping lines and non-Lefschetz locus}
In this section describe the elements of the non-Lefschetz locus of $M$. In particular we want to extend to our setting the connection between splitting type and Lefschetz elements first made by Brenner and Kaid \cite{3} for Artinian algebras with stable syzygy bundle. 

We first observe that a linear form  $\ell$ is a Weak Lefschetz element for $M={H}_*^1(\Pj^2, \E)$ (Definition  \ref{WLP}) if and only if is a Weak Lefschetz element for ${H}_*^1(\Pj^2, \E(m))$, for any integer $m$. Then for this section we will assume, when it is convenient, that $\E$ is normalized  and so the first Chern Class $c_1(\E)\in\{0,-1\}$ (\cite{OSS}).
The Grauert-M\"{u}lich theorem (\cite{OSS}) together with \cite[Proposition 3.5]{FFP21} give us the following:
\begin{teo} \label{GM}
	Let $\E$ be a normalized rank $2$ vector bundle on $\Pj^2$ and let $\ell$ be a general line. Then
	\begin{itemize}
		\item if $\E$ is semistable and $c_1(\E)=0$, then $\E_{|\ell}\cong\mathcal{O}_{\mathbb{P}^1}\oplus\mathcal{O}_{\mathbb{P}^1}$;
		\item if $\E$ is semistable and $c_1(\E)=-1$, then $\E_{|\ell}\cong\mathcal{O}_{\mathbb{P}^1}(-1)\oplus\mathcal{O}_{\mathbb{P}^1}$;
		\item if $\E$ is unstable, with instability index $k$ and $c_1(\E)=0$, then $k>0$ and $\E_{|\ell}\cong\mathcal{O}_{\mathbb{P}^1}(-k)\oplus\mathcal{O}_{\mathbb{P}^1}(k)$;
		\item if $\E$ is unstable, with instability index $k$ and $c_1(\E)=-1$, then $k\geq0$ and $\E_{|\ell}\cong\mathcal{O}_{\mathbb{P}^1}(-k-1)\oplus\mathcal{O}_{\mathbb{P}^1}(k)$.
	\end{itemize} 
\end{teo} 
\begin{df}
	Given a line $\ell$ the \emph{splitting type} of $\E$ on $\ell$ is the couple $(a,b)$ where $\E_{|\ell}\cong\mathcal{O}_{\mathbb{P}^1}(a)\oplus\mathcal{O}_{\mathbb{P}^2}(b)$.
\end{df}
\begin{rmk}
	For $\E$ semistable, not necessarily normalized, Theorem \ref{GM} is equivalent to say that a general line $\ell$, has splitting type $(a,b)$  with $|a-b|\leq 1$.
\end{rmk}
\begin{df}
	A \emph{jumping line} is a linear element $\ell$ with spitting type different by the splitting type of as a general line.
\end{df}
In \cite[Theorem 3.7]{FFP21}, the authors  proved that a general line is a Weak Lefschetz element for $M$ using the splitting type of those lines. Our main result in this section is prove that such a condition is also necessary, so we have the following theorem:
\begin{teo}\label{jum}
The non-Lefschetz locus is the set of the jumping lines.
\end{teo}
\begin{proof}
	We want to prove that given a line $\ell$ with spitting type $(a,b)$, $\ell$ is a weak Lefschetz element if and only if $(a,b)$ is the splitting type on the general line. 
	
	Before proceeding, recall that a linear element $\ell$ is a Weak Lefschetz element if and only 
	$\times \ell: M_{t-1}\to M_{t}$ has maximum rank for all $t$, if and only if
	$\mu_\ell: {H}^1(\Pj^2, \E(t-1))\to{H}^1(\Pj^2, \E(t))$ is injective or surjective.
	Consider the short exact sequence
	$$0\to\E(t-1)\to\E(t)\to\E(t)_{|\ell}\to0.$$
	Applying the global section functor we get the long exact sequence
	\begin{align*}
		0\to H^0(\Pj^2,\E(t-1))\to H^0(\Pj^2,\E(t))\to H^0(\Pj^2,\E(t)_{|\ell})\to H^1(\Pj^2,\E(t-1))\stackrel{\mu_\ell}{\to} H^1(\Pj^2,\E(t))\to \\
		\to H^1(\Pj^2,\E(t)_{|\ell})\to H^2(\Pj^2,\E(t-1))\to H^2(\Pj^2,\E(t))\to H^2(\Pj^2,\E(t)_{|\ell})=0.
	\end{align*}
	Then we have the following facts, which we will strongly use in our arguments.
	\begin{enumerate}
		\item \label{A1} $\mu_\ell$ is injective if $h^0(\Pj^2,\E(t)_{|\ell})=0$;
		\item \label{A2} $\mu_\ell$ is injective if and only if $h^0(\Pj^2,\E(t)_{|\ell})- h^0(\Pj^2,\E(t))+h^0(\Pj^2,\E(t-1))=0$;
		\item \label{A3} $\mu_\ell$ is injective if and only if the map $  H^0(\Pj^2,\E(t))\to H^0(\Pj^2,\E(t)_{|\ell})$ is surjective;
		\item \label{B1} $\mu_\ell$ is surjective if $h^1(\Pj^2,\E(t)_{|\ell})=0$;
			\item \label{B2} $\mu_\ell$ is surjective if and only if $h^1(\Pj^2,\E(t)_{|\ell})- h^2(\Pj^2,\E(t-1))+h^2(\Pj^2,\E(t))=0$;
		\item \label{B3} $\mu_\ell$ is surjective if and only if the map $  H^1(\Pj^2,\E(t)_{|\ell})\to H^2(\Pj^2,\E(t-1))$ is injective.
	\end{enumerate}
We argue in two different cases depending on whether $\E$ is semistable or unstable.

\emph{Case 1}.	
Assume $\E$ semistable. If $\ell$ is not a jumping line then $\E_{|\ell}\cong\mathcal{O}_{\mathbb{P}^1}(a)\oplus\mathcal{O}_{\mathbb{P}^2}(b)$ with $|a-b|\leq 1$. Shifting the sequence we get
	$$\E(t)_{|\ell}\cong\mathcal{O}_{\mathbb{P}^1}(t+a)\oplus\mathcal{O}_{\mathbb{P}^1}(t+b).$$ Without loss of generality we assume $a\leq b$. Then 
	$h^0(\Pj^2,\E(t)_{|\ell})\cong h^0(\Pj^2,\mathcal{O}_{\mathbb{P}^1}(t+a)\oplus\mathcal{O}_{\mathbb{P}^2}(t+b))=0$ if $t<-b$.
	Using Serre Duality (\cite{Hartshorne}) $$h^1(\Pj^2,\E(t)_{|\ell})=h^0(\Pj^2,\E^\vee(-t-2)_{|\ell})\cong h^0(\Pj^2,\mathcal{O}_{\mathbb{P}^1}(-t-a-2)\oplus\mathcal{O}_{\mathbb{P}^1}(-t-b-2))$$
	and $h^0(\Pj^2,\E(t)_{|\ell})=0$ when $t\geq-a-1$.
	Then by Fact (\ref{A1}) the map $\mu_\ell: {H}^1(\Pj^2, \E(t-1))\to{H}^1(\Pj^2, \E(t))$ is injective when $t<-b$  and, by Fact (\ref{B1})  it is surjective when $t\geq -a-1$. Since $-a-1\leq -b$ ($a\leq b\leq a+1$), we obtain that $\ell$ is a Weak Lefschetz element.
	
	Now assume that $\ell$ is a jumping line, so its splitting type is $(a,b)$ with $b-a>1$ (without loss of generality $a<b$).
	Using the fact that $\E(t)_{|\ell}\cong\mathcal{O}_{\mathbb{P}^1}(t+a)\oplus\mathcal{O}_{\mathbb{P}^1}(t+b)$, we have that $c_1(\E(t))=(t+a)+(t+b)$. In particular
	$c_1(\E(-b))<0$ and $c_1(\E(-a))>0$. Then we can find an integer $m$ such that $-b\leq m \leq -a-2$ with $c_1(\E(m))<0$ and $c_1(\E(m+2))>0$.
	Since $\E$ is semistable, $\E(m)$ is semistable. Now by \cite[Corollary 2.7]{1} we have $h^0(\Pj^2,\E(m)(t))=0$ for  any $t< -c_1(\E(m))/2$. $c_1(\E(m))<0$ so for $t=0$, $h^0(\Pj^2,\E(m))=0$.
	Since $m+b\geq 0$, $$H^0(\Pj^2,\E(m)_{|\ell})\cong H^0(\Pj^2,\mathcal{O}_{\mathbb{P}^1}(m+a)\oplus \mathcal{O}_{\mathbb{P}^1}(m+b))\neq 0.$$ Then the map 
	$  0=H^0(\Pj^2,\E(m))\to H^0(\Pj^2,\E(m)_{|\ell})$ cannot be surjective and so by Fact (\ref{A3})
	$\mu_\ell: {H}^1(\Pj^2, \E(m-1))\to{H}^1(\Pj^2, \E(m))$
	is not injective.
	
	Now by Serre Duality ${H}^2(\Pj^2, \E(m-1))\cong{H}^0(\Pj^2, \E^\vee(-m-2))$. $\E^\vee(-m-2)$ is semistable, because $\E$ is semistable and
	$c_1(\E^\vee(-m-2))=-c_1(\E(m+2))<0$. Using \cite[Corollary 2.7]{1} as above we get that ${H}^2(\Pj^2, \E(m-1))\cong{H}^0(\Pj^2, \E^\vee(-m-2))=0$.
	Now using again Serre Duality
	\begin{align*}
	H^1(\Pj^2,\E(m)_{|\ell})&\cong H^1(\Pj^2,\mathcal{O}_{\mathbb{P}^1}(m+a)\oplus \mathcal{O}_{\mathbb{P}^1}(m+b)) \\ &\cong H^0(\Pj^2,\mathcal{O}_{\mathbb{P}^1}(-m-a-2)\oplus \mathcal{O}_{\mathbb{P}^1}(-m-b-2))\neq 0	
	\end{align*}
 	because $-a-m-2\geq 0$ ($m\leq-a-2$).
	Then the map $ H^1(\Pj^2,\E(m)_{|\ell})\to H^2(\Pj^2,\E(m-1))=0$ is not injective and so $$\mu_\ell: {H}^1(\Pj^2, \E(m-1))\to{H}^1(\Pj^2, \E(m))$$
	is not surjective by Fact (\ref{B3}).
	So we have found an integer $m$ such that $\mu_\ell: {H}^1(\Pj^2, \E(m-1))\to{H}^1(\Pj^2, \E(m))$ does not have maximum rank. Thus $\ell$ is not a Lefschetz element.
	
	\emph{Case 2}.	
	We now assume $\E$ unstable with instability index $k$. Without loss of generality we can assume that $\E$ is normalized and we distinguish 2 sub-cases depending on the value of the first Chern class of $\E$. 
	
	\emph{Case 2.1}.
	We first treat the case $c_1(\E)=0$. Then we know that $k>0$ and $\E^\vee=\E$.
	Let $\ell$ be line such that $\E_{|\ell}\cong\mathcal{O}_{\mathbb{P}^1}(a)\oplus\mathcal{O}_{\mathbb{P}^1}(b)$, and without loss of generality we can assume $a\leq b$.	We want to prove that $\ell$ is a Lefschetz element if and only if $b=k$ and $a=-k$. 
	Since $a+b=c_1(\E)=0$ we have $a=-b$ for $b\geq 0$.

	By \cite[Proposition 3.6]{FFP21} for $t<k$ we have 
	\begin{align*}
		 &h^0(\Pj^2,\E(t)_{|\ell})- h^0(\Pj^2,\E(t))+h^0(\Pj^2,\E(t-1))\\
		&=h^0(\Pj^2,\mathcal{O}_{\mathbb{P}^1}(t-b))+ h^0(\Pj^2,\mathcal{O}_{\mathbb{P}^1}(t+b))-\binom{k+t+2}{2}+\binom{k+t+1}{2}\\
		&=\begin{cases}
			2t+2  & \text{ if }  t\geq b;\\
			b+t+1  & \text{ if } -b\leq t<b;\\
			0  & \text{ if } t<-b.
		\end{cases}+\begin{cases}
			-k-t-1 & \text{ if } -k\leq t<k;\\
			0  & \text{ if } t<-k;
		\end{cases} 
	\end{align*}
 	In a similar way using Serre Duality, the fact that $\E=\E^\vee$, and \cite[Proposition 3.7]{FFP21}, we have for $-t-2< k$ 
 		\begin{align*}
 		&h^1(\Pj^2,\E(t)_{|\ell})- h^2(\Pj^2,\E(t-1))+h^2(\Pj^2,\E(t))\\
 		&=h^1(\Pj^2,\mathcal{O}_{\mathbb{P}^1}(t-b))+ h^1(\Pj^2,\mathcal{O}_{\mathbb{P}^1}(t+b))- h^0(\Pj^2,\E(-t-3))+h^0(\Pj^2,\E(-t-2))\\
 		&=h^0(\Pj^2,\mathcal{O}_{\mathbb{P}^1}(b-t-2))+ h^0(\Pj^2,\mathcal{O}_{\mathbb{P}^1}(-b-t-2))-\binom{k-t-1 }{2}-\binom{k-t}{2}\\
 		&= \begin{cases}
 			  0& \text{ if }  t> b-2;\\
 			b-t-1  & \text{ if } -b-2< t\leq b-2;\\
 			-2t-2  & \text{ if } t\leq-b-2.
 		\end{cases}+\begin{cases}
 			0 & \text{ if } t>k-2;\\
 			 -k+t+1 & \text{ if }-k-2<t\leq k-2 ;
 		\end{cases}\\ 
 	\end{align*}
	There are only 3 possible cases in terms of the values of $a,b$ and $k$.
	\begin{itemize}
		\item $-k<a=-b\leq b<k$;
		\item $-k=a=-b< b=k$ (since $k>0$);
		\item $a=-b<-k<k<b$.
	\end{itemize}
\emph{Case 2.1.1}. 
$-k<a=-b\leq b<k$. We have
\begin{align*}
	h^0(\Pj^2,\E(t)_{|\ell})- h^0(\Pj^2,\E(t))+h^0(\Pj^2,\E(t-1))=\begin{cases}
		-k+t+1 & \text{ if } b\leq t<k;\\
		-k+b  & \text{ if } -b\leq t<b;\\
		-k-t-1 & \text{ if } -k\leq t<-b;\\
		0  & \text{ if } t<-k.
	\end{cases} 
\end{align*}
Then using Fact (\ref{A2}), we can assure that
 $\mu_\ell$ is not injective for $-k\leq t \leq k-2$. Moreover, since $k>0$ the interval $[-k, k-2]$ is not empty.
Similarly 
$$h^1(\Pj^2,\E(t)_{|\ell})- h^2(\Pj^2,\E(t-1))+h^2(\Pj^2,\E(t))= \begin{cases}
	0& \text{ if }  t> k-2;\\
	-k+t+1  & \text{ if } b-2< t\leq k-2;\\
	b-k  & \text{ if } -b-2< t\leq b-2;\\
	-k-t-1 & \text{ if }-k-2<t\leq -b-2 ;
\end{cases}$$
and using Fact (\ref{B2}), we obtain that $\mu_\ell$ is  not surjective for $-k\leq t\leq k-2$.

We have proved that the map $\mu_\ell$ does not have maximum rank for $-k\leq t\leq k-2$. So $\ell$ is not a Weak Lefschetz element.

\emph{Case 2.1.2}. 
$-k=a=-b< b=k$.
Here have that $h^0(\Pj^2,\E(t)_{|\ell})- h^0(\Pj^2,\E(t))+h^0(\Pj^2,\E(t-1))=0$ for $t<k$ as well as $h^1(\Pj^2,\E(t)_{|\ell})- h^2(\Pj^2,\E(t-1))+h^2(\Pj^2,\E(t))=0$ for $t>-k-2$. Then by Facts (\ref{A2})  and 
(\ref{B2}) 
\begin{itemize}
	\item $\mu_\ell$ is injective for $t\leq k-1$;
	\item $\mu_\ell$ is surjective for $t\geq -k-1$
\end{itemize}
so it always has maximum rank, since $k>0$, and in this case $\ell$ is a Weak Lefschetz element.

\emph{Case 2.1.3}. 
$a=-b<-k<k<b$. In this case, we have that
\begin{align*}
	h^0(\Pj^2,\E(t)_{|\ell})- h^0(\Pj^2,\E(t))+h^0(\Pj^2,\E(t-1))=\begin{cases}
		-k+b & \text{ if } -k\leq t<k;\\
		b+t+1  & \text{ if } -b\leq t<-k;\\
		0  & \text{ if } t<-b.
	\end{cases} 
\end{align*}
Then using Fact (\ref{A2}) we conclude that the map $\mu_\ell$ is not injective for $-b\leq t \leq k-1$.
To check when the map $\mu_\ell$ is surjective  we use 
$$h^1(\Pj^2,\E(t)_{|\ell})- h^2(\Pj^2,\E(t-1))+h^2(\Pj^2,\E(t))= \begin{cases}
	0& \text{ if }  t> b-2;\\
	b-t-1  & \text{ if } k-2< t\leq b-2;\\
	b-k  & \text{ if } -k-2< t\leq k-2;
\end{cases}$$
and  Fact (\ref{B2}) assure that $\mu_\ell$ is  not surjective for $-k-1\leq t\leq b-2$.

We have proved that the map $\mu_\ell$ does not have maximum rank for $-k-1\leq t\leq k-1$. Moreover, the interval $[-k-1,k-1]$ is not empty because $k>0$. So $\ell$ is not a Weak Lefschetz element.

Therefore, we can conclude that for  $\E$ unstable, normalized with $c_1(\E)=0$, $\ell$ is a Weak Lefschetz element if and only the splitting type of $\ell$ is $(k,-k)$, i.e. if and only if $\ell$ is not a jumping line.

\emph{Case 2.2}. Now we assume  $\E$ unstable, normalized with $c_1(\E)=-1$ and instability index $k$, and we proceed in an analogous way. In this case we have that $k\geq0$ (instability index) and $\E^\vee=\E(1)$.
Let $\ell$ be a linear element, with splitting type $(a,b)$, without loss of generality we can assume $a\leq b$ as above. Since $a+b=c_1(\E)=-1$ here we have $a=-b-1<0$ for $b\geq 0$.
Recall that we want to prove that $\ell$ is a Lefschetz element if and only if it has the same splitting type of a general line, i.e. $b=k$ and $a=-k-1$. 

By \cite[Proposition 3.7]{FFP21} for $t\leq k$ we have 
\begin{align*}
	&h^0(\Pj^2,\E(t)_{|\ell})- h^0(\Pj^2,\E(t))+h^0(\Pj^2,\E(t-1))\\
	&=\begin{cases}
		-k-t-1 & \text{ if } -k\leq t\leq k;\\
		0  & \text{ if } t<-k;
	\end{cases} + \begin{cases}
		2t+1  & \text{ if }  t> b;\\
		b+t+1  & \text{ if } -b\leq t\leq b;\\
		0  & \text{ if } t<-b.
	\end{cases}\\
\end{align*}
In a similar way using Serre Duality, the fact that $\E^\vee=\E(1)$, and \cite[Proposition 3.7]{FFP21}, for $-t-1\leq k$, i.e. $t>-k-2$, we have
\begin{align*}
	&h^1(\Pj^2,\E(t)_{|\ell})- h^2(\Pj^2,\E(t-1))+h^2(\Pj^2,\E(t))\\
	&=h^1(\Pj^2,\mathcal{O}_{\mathbb{P}^1}(t-b-1))+ h^1(\Pj^2,\mathcal{O}_{\mathbb{P}^1}(t+b))- h^0(\Pj^2,\E^\vee(-t-2))+h^0(\Pj^2,\E^\vee(-t-3))\\
	&=h^0(\Pj^2,\mathcal{O}_{\mathbb{P}^1}(b-t-1))+ h^0(\Pj^2,\mathcal{O}_{\mathbb{P}^1}(-b-t-2))- h^0(\Pj^2,\E(-t-1))+h^0(\Pj^2,\E(-t-2))\\
	&= \begin{cases}
		0& \text{ if }  t\geq b;\\
		b-t  & \text{ if } -b-2< t<b;\\
		-2t-1  & \text{ if } t\leq-b-2.
	\end{cases}+\begin{cases}
		0 & \text{ if } t\geq k;\\
		-k+t & \text{ if }-k-2<t< k .
	\end{cases}\\ 
\end{align*}

We consider separately the three possible sub-cases depending on the values of $a,b$ and $k$. In each sub-case we study when the map $\mu_\ell$ is injective and when it is surjective using Facts (\ref{A2}) and (\ref{B2}), but since the argument is the same one used in \emph{Case 2.1} we omit the computations.

\emph{Case 2.2.1}.
 $-k\leq a=-b-1<0\leq b<k$. 
 Here that the map $\mu_\ell$ does not have maximum rank for $-k\leq t\leq k-1$. Moreover, the interval $[-k,k-1]$ is not empty because $k>b\geq 0$. So $\ell$ is not a Weak Lefschetz element.

\emph{Case 2.2.2}. 
$k=b$ and $a=-k-1$.
In this case $\mu_\ell$  always has maximum rank, so $\ell$ is a Weak Lefschetz element.

\emph{Case 2.2.3}. 
$a=-b-1<-k\leq k <b$. 
Here, the map $\mu_\ell$ does not have maximum rank for $-k-1\leq t\leq k$ and the interval $[-k-1, k]$ is not empty because $k\geq0$. So $\ell$ is not a Weak Lefschetz element.

This proves that for  $\E$ unstable, normalized with $c_1(\E)=-1$, $\ell$ is a Weak Lefschetz element if and only if it is not a jumping line.

This concludes our proof: $\ell$ is in the non-Lefschetz locus if only if it is not a Weak Lefschetz element if and only if $\ell$ is a jumping line. \end{proof}

\section{Expected codimension in the case $b_i=0$ for all $i=1,\dots,n$}
Now we have proved that the elements of the non-Lefschetz locus are exactly the jumping lines so we will use this to compute the codimension of the non-Lefschetz locus.
In this section we will consider the sub-case when all the $b_i$ are zero.
In this case we have $d=a_1+\dots+a_{n+2}$ and we have the exact sequence
$$0\to E \to \bigoplus_{i=1}^{n+2}R(-a_i)\to\bigoplus_{i=1}^{n} R \to M \to 0.$$
By \cite[Lemma 3.1, Proposition 3.9]{F19}, $M$ is a level module and Symmetrically Gorenstein of socle degree $d-3$. Hence  by  \cite[Corollary 5.12]{F19}, the non-Lefschetz locus is the same as the non-Lefschetz locus at the middle degree $\lfloor \frac{e-1}{2}\rfloor=\lfloor \frac{d-4}{2}\rfloor$:
$$\loc_M=\loc_{\lfloor \frac{d-4}{2}\rfloor,M}.$$
Then in this case we can consider expected codimension of the non-Lefschetz locus and the main result of this section will be proving that in the general case such expected codimension is achieved. 

Recall that the expected codimension of the non-Lefschetz locus in degree $it$ is $h_{t+1}-h_t+1$, where $h_t$ is the Hilbert Function, and that if such dimension is achieved then the degree is $\deg \loc_I=\binom{h_{t+1}}{h_t-1}$.
Since in this case $\loc_M=\loc_{\lfloor \frac{d-4}{2}\rfloor,M}$, the expected codimension of the non-Lefschetz locus is 
$$\exc (\loc_M)=h_{\lfloor \frac{d-4}{2}\rfloor+1}-h_{\lfloor \frac{d-4}{2}\rfloor}+1=h_{\lfloor \frac{d}{2}\rfloor-1}-h_{\lfloor \frac{d}{2}\rfloor-2}+1.$$
\begin{prop} \label{b} Let $\E$ be a general rank $2$ vector bundle on $\Pj^2$ such that $M={H}_*^1(\Pj^2, \E)$ is a finite length module that can be expressed as cokernel of a map $\varphi:\bigoplus_{i=1}^{n+2}R(-a_n)\to\bigoplus_{i=1}^{n} R$
Then the non-Lefschetz locus of $M$ has the expected codimension.
\end{prop}
\begin{proof}
	We want to prove that the codimension of the non-Lefschetz locus is $h_{\lfloor \frac{d}{2}\rfloor-1}-h_{\lfloor \frac{d}{2}\rfloor-2}+1$. We proceed separately for the case when $d$ is odd and $d$ is even.

\emph{Case 1}. 
Assume $d$ even. Then the socle degree $e=d-3$ is odd and by the symmetry of the Hilbert function of $M$ we have 
$$\exc (\loc_M)=h_{\frac{d}{2}-1}-h_{\frac{d}{2}-2}+1=h_{ \frac{e+1}{2}}-h_{\frac{e-1}{2}}+1=1.$$
By \cite[Theorem  3.7]{FFP21}, $M$ has the weak Lefschetz property, so $\cod \loc_M> 0$ and we get:
 $$0< \cod \loc_M\leq \exc (\loc_M)=1.$$
Therefore for $d$ even the non-Lefschetz locus always has the expected codimension and so $\loc_M$ is a hypersurface in $\Pj^2$ of degree  $$\deg \loc_I=\binom{h_{\frac{d}{2}-1}}{h_{\frac{d}{2}-2}-1}+\binom{h_{\frac{d}{2}-1}}{h_{\frac{d}{2}-1}-1}=h_{\frac{d}{2}-1}.$$

\emph{Case 2}.
Let us assume now on that $d$ is odd. The socle degree $e=d-3$ is even.
Using the exact sequence (\ref{1})
$$0\to \mathcal{E} \to \bigoplus_{i=1}^{n+2}\mathcal{O}_{\mathbb{P}^2}(-a_n)\to\bigoplus_{i=1}^{n} \mathcal{O}_{\mathbb{P}^2} \to 0$$
we can compute the Chern classes of $\E$ 
$$c_1(\E)=-a_1-\dots -a_{n+2}=-d \ \ \ \ \ \ \ \ c_2(\E)=\sum_{i\neq j}a_ia_j.$$
Similarly using the shifted sequence  
$$0\to \mathcal{E}(t) \to \bigoplus_{i=1}^{n+2}\mathcal{O}_{\mathbb{P}^2}(t-a_n)\to\bigoplus_{i=1}^{n} \mathcal{O}_{\mathbb{P}^2}(t) \to 0$$
we can compute the Chern classes of $\E(t)$ 
$$c_1(\E(t))=-d+2t=c_1(\E)+2t\ \ \ \ \ \ \ \ c_2(\E)=c_2(\E)+c_1(\E)t+t^2$$
and by the Riemann–Roch theorem the Euler characteristic is
\begin{equation}\label{chi}
	\chi(\E(t))=2+\frac{2}{3}c_1(\E)+\frac{1}{2}c_1(\E)^2-c_2(\E)+(3+c_1(\E))t+t^2.
\end{equation}
Since we are assuming $d$ odd the normalized vector bundle is $\En=\E(\frac{d-1}{2})$ and $c_1(\En)=-1.$
Then $\En$ is stable if and only if $\En$ semistable if and only if $H^0(\Pj^2, \En)=0$ by  \cite[Lemma 3.2]{FFP21}.
Since $b_i=0$ the sequence (\ref{2}) becomes
$$0\to \bigoplus_{i=1}^{n}\mathcal{O}_{\mathbb{P}^2}(-d)\to \bigoplus_{i=1}^{n+2} \mathcal{O}_{\mathbb{P}^2}(-d+a_i) \to \E\to 0$$
and shifting by $\frac{d-1}{2}$ we obtain 
$$0\to \bigoplus_{i=1}^{n}\mathcal{O}_{\mathbb{P}^2}\left(-\frac{d+1}{2}\right)\to \bigoplus_{i=1}^{n+2} \mathcal{O}_{\mathbb{P}^2}\left(-\frac{d+1}{2}+a_i\right) \to\En \to 0.$$
Hence, $\En$ has no global section if and only if $-\frac{d+1}{2}+a_{n+2}<0$, i.e. $a_{n+2}<a_1+\dots + a_{n+1}+1$.
Therefore, for $d$ odd we have $\E$ stable  for $a_{n+2}<a_1+\dots+ a_{n+1}+1$. However, since $d$ is odd it is impossible to have $a_{n+2}=a_1+\dots+ a_{n+1}$, so we can rewrite the condition as  $a_{n+2}<a_1+\dots+ a_{n+1}$ and we obtain that $\E$ is unstable  for $a_{n+2}> a_1+\dots +a_{n+1}$.

First let us compute the expected codimension:
\begin{align*}
	\exc \left(\loc_M\right)=& \ h_{\lfloor \frac{d}{2}\rfloor-1}-h_{\lfloor \frac{d}{2}\rfloor-2}+1=h_{\frac{d-3}{2}}-h_{ \frac{d-5}{2}}+1\\
	=& \ h^1\left(\Pj^2,\E\left(\frac{d-3}{2}\right)\right)-h^1\left(\Pj^2,\E\left(\frac{d-5}{2}\right)\right)\\
	=& \ -\chi\left(\E\left(\frac{d-3}{2}\right)\right)+h^0\left(\Pj^2,\E\left(\frac{d-3}{2}\right)\right)+h^2\left(\Pj^2,\E\left(\frac{d-3}{2}\right)\right)\\&+\chi\left(\E\left(\frac{d-5}{2}\right)\right)-h^0\left(\Pj^2,\E\left(\frac{d-5}{2}\right)\right)-h^2\left(\Pj^2,\E\left(\frac{d-5}{2}\right)\right)+1\\
	=& \ \chi\left(\E\left(\frac{d-5}{2}\right)\right)-\chi\left(\E\left(\frac{d-3}{2}\right)\right)+1\\&-h^0\left(\Pj^2,\E\left(\frac{d-1}{2}\right)\right)+2h^0\left(\Pj^2,\E\left(\frac{d-3}{2}\right)\right)-h^0\left(\Pj^2,\E\left(\frac{d-5}{2}\right)\right)\\
\end{align*}
where we have used that $c_1(\En)=-1$ so $\En^\vee=\En(1)$, which implies
\begin{align*}
\E^\vee&=\left(\En\left(-\frac{d-1}{2}\right)\right)^\vee=\En^\vee\left(\frac{d-1}{2}\right)=\En(1)\left(\frac{d-1}{2}\right)\\&=\E\left(\frac{d-1}{2}\right)\left(\frac{d-1}{2}+1\right)=\E(d),
\end{align*}
and by Serre Duality
$$h^2\left(\Pj^2,\E\left(\frac{d-3}{2}\right)\right)=h^0\left(\Pj^2,\check{\E}\left(-\frac{d-3}{2}-3\right)\right)=h^0\left(\Pj^2,\E\left(d-\frac{d+3}{2}\right)\right)=h^0\left(\Pj^2,\E\left(\frac{d-3}{2}\right)\right),$$
$$h^2\left(\Pj^2,\E\left(\frac{d-5}{2}\right)\right)=h^0\left(\Pj^2,\check{\E}\left(-\frac{d-5}{2}-3\right)\right)=h^0\left(\Pj^2,\E\left(d-\frac{d+1}{2}\right)\right)=h^0\left(\Pj^2,\E\left(\frac{d-1}{2}\right)\right).$$
Now using (\ref{chi}) we have
\begin{align*}
&\chi\left(\E\left(\frac{d-5}{2}\right)\right)-\chi\left(\E\left(\frac{d-3}{2}\right)\right)+1\\
&=(3+c_1(\E))\left(\frac{d-5}{2}\right)+\left(\frac{d-5}{2}\right)^2-(3+c_1(\E))\left(\frac{d-3}{2}\right)-\left(\frac{d-3}{2}\right)^2+1\\
&=(3-d)\left(\frac{d-5}{2}-\frac{d-3}{2}\right)+\left(\frac{d-5}{2}\right)^2-\left(\frac{d-3}{2}\right)^2+1=2.
\end{align*}
To compute $-h^0\left(\Pj^2,\E\left(\frac{d-1}{2}\right)\right)+2h^0\left(\Pj^2,\E\left(\frac{d-3}{2}\right)\right)-h^0\left(\Pj^2,\E\left(\frac{d-5}{2}\right)\right)$ we need to consider the cases $\E$ stable and $\E$ unstable separately.

\emph{Case 2.1}
 $\E$ stable. Here we have $a_{n+2}<a_1+\dots +a_{n+1}$, and $h^0(\Pj^2,\E(t))=0$ for any $t\leq -c_1(\E)/2=d/2$ by  \cite[Corollary 2.7]{1}; in particular 
$$h^0\left(\Pj^2,\E\left(\frac{d-1}{2}\right)\right)=h^0\left(\Pj^2,\E\left(\frac{d-3}{2}\right)\right)=h^0\left(\Pj^2,\E\left(\frac{d-5}{2}\right)\right)=0$$
and so $$\exc\left(\loc_M\right)=\chi\left(\E\left(\frac{d-5}{2}\right)\right)-\chi\left(\E\left(\frac{d-3}{2}\right)\right)+1+0=2.$$

We proved in Proposition \ref{jum} that the non-Lefschetz locus is the set of the jumping lines. Since we are assuming that $\E$ is general, $c_1(\En)=-1$, and $\E$ stable, by \cite[Corollary 10.7.1]{16} we can conclude that we have exactly $\binom{c_2(\En)}{2}$ jumping lines. This proves that in the case $d$ odd and $a_{n+2}<a_1+\dots+ a_{n+1}$, the expected codimension is achieved.
 
 \emph{Case 2.2}
 $\E$ unstable. In this case $d$ is odd, $c_1(\En)=-1$ and $a_{n+2}\geq a_1+\dots a_{n+1}$. 

Shifting by $t+\frac{d-1}{2}$ the sequence (\ref{2})
$$0\to \bigoplus_{i=1}^{n+2}\mathcal{O}_{\mathbb{P}^2}\left(t-\frac{d+1}{2}\right)\to \bigoplus_{i=1}^{n} \mathcal{O}_{\mathbb{P}^2}\left(t-\frac{d+1}{2}+a_i\right) \to\En(t) \to 0$$
we obtain $H^0(\Pj^n, \En(t))=0$ if and only if $t\leq a_{n+2}-\frac{d+1}{2}$, so the instability index of $\En$ is $k=a_{n+2}-\frac{d+1}{2}$.
Then using  \cite[Proposition 3.6]{FFP21} 
\begin{align*}
	h^0(\Pj^n, \E(t))=&h^0\left(\Pj^n, \En\left(t-\frac{d-1}{2}\right)\right)=\binom{k+t-\frac{d-1}{2}+2}{2}\\=&\binom{a_{n+2}-\frac{d+1}{2}+t-\frac{d-1}{2}+2}{2}=\binom{a_{n+2}-d+t+2}{2}.
\end{align*}
In particular 
\begin{align*}
	h^0\left(\Pj^2,\E\left(\frac{d-1}{2}\right)\right)&=\binom{a_{n+2}-\frac{d+1}{2}+2}{2}\\
	h^0\left(\Pj^2,\E\left(\frac{d-3}{2}\right)\right)&=\binom{a_{n+2}-\frac{d+3}{2}+2}{2}=\binom{a_{n+2}-\frac{d+1}{2}+1}{2}\\
	h^0\left(\Pj^2,\E\left(\frac{d-5}{2}\right)\right)&=\binom{a_{n+2}-\frac{d+5}{2}+2}{2}=\binom{a_{n+2}-\frac{d+1}{2}}{2}
\end{align*}
and finally we can compute the expected codimension 
\begin{align*}
\exc \left(\loc_M\right)=&\ \chi\left(\E\left(\frac{d-5}{2}\right)\right)-\chi\left(\E\left(\frac{d-3}{2}\right)\right)+1\\&-h^0\left(\Pj^2,\E\left(\frac{d-1}{2}\right)\right)+2h^0\left(\Pj^2,\E\left(\frac{d-3}{2}\right)\right)-h^0\left(\Pj^2,\E\left(\frac{d-5}{2}\right)\right)\\
=&\ 2  -\binom{a_{n+2}-\frac{d+1}{2}+2}{2}+ 2\binom{a_{n+2}-\frac{d+1}{2}+1}{2}- \binom{a_{n+2}-\frac{d+1}{2}}{2}\\
=&\ 2 -\left(a_{n+2}-\frac{d+1}{2}+1\right)+\left(a_{n+2}-\frac{d+1}{2}\right)=1.
\end{align*}
Since $M$ has the weak Lefschetz property by \cite[Theorem  3.7]{FFP21}, $\cod \loc_M> 0$.
Then $$0< \cod \loc_M\leq \exc (\loc_M)=1$$ and so the expected codimension is achieved also in this last case.\end{proof}
\begin{rmk}
	The hypothesis of generality is necessary for the case $d=a_1,\dots +a_{n+2}$ odd and $a_{n+2}\leq a_1,\dots +a_{n+1}-1$. Even in the case of the complete intersection ($n=1$) the result does not hold without that hypothesis.
	For example if we consider the monomial complete intersection $M=R/(x_1^3,x_2^4, x_3^4)$ the expected codimension is $2$, since $d=11$ and $a_3=4< 6=a_1+a_2-1$, but we know from  \cite[Corollary 3.4]{main} that $\cod(\loc_M)=1$.
	In the other cases the hypothesis of generality is unnecessary since 
	$$\exc(\loc_M)=\begin{cases} 1& \text{ if } d \text{ even;}\\
		1& \text{ if } d \text{ odd and }a_{n+2}> a_1,\dots +a_{n+1};\\
	\end{cases}$$
and by  \cite[Theorem  3.7]{FFP21}, $M$  has the weak Lefschetz property,  so $\cod \loc_M\geq 1$.	  
\end{rmk}

\section{Codimension of the non-Lefschetz locus of general vector bundles of rank $2$ over $\Pj^2$}
In this section we want to generalize the results of the previous section for any finite length $M={H}_*^1(\Pj^2, \E)$.
First we will prove that the non-Lefschetz locus is concentrated in the middle degree. This will allow us to consider the expected codimension of the whole non-Lefschetz locus. Finally the main result of this section will be that for a general rank $2$ vector bundle $\E$ the non-Lefschetz locus has the expected codimension.

Recall that we have the exact sequence
$$0\to E \to \bigoplus_{i=1}^{n+2}R(-a_i)\to\bigoplus_{i=1}^{n} R(-b_i) \to M \to 0.$$
In this case $M$ is not a level module, but we can still prove the following:
\begin{prop}\label{new}
	The non-Lefschetz locus of $M$ is the same as the non-Lefschetz locus in the middle degree:
	$$\loc_M =\loc_{\lfloor \frac{d-4}{2}\rfloor,M}$$
\end{prop}
\begin{proof}
	We start with proving that $\loc_M =\loc_{\lfloor \frac{d-4}{2}\rfloor-1,M} \cup \loc_{\lfloor \frac{d-4}{2}\rfloor,M}.$
	
	We first can notice that if a finite length module $M$ has unimodal Hilbert function $\dots\leq h_{m-1}\leq h_m\geq h_{m+1}\geq\dots$, $[\socle (M)]_i=0$ for $i<m$, and it does not have generators in degree $i>m$, then $\loc_M=\loc_{M,m}\cup\loc_{M,m-1}$.
	In fact by \cite[Proposition 2.6]{main} if $h_i\leq h_{i+1}\leq h_{i+2}$ and $[\socle (M)]_i=0$ then $I(\loc_{M,i+1})\subseteq I(\loc_{M,i})$, i.e. $\loc_{M,i}\subseteq\loc_{M,i+1}$. 
	Dualizing this we also have that if $h_i\geq h_{i+1}\geq h_{i+2}$ and there are no generators of degree $i$ then $I(\loc_{M,i+1})\supseteq I(\loc_{M,i})$, i.e. $\loc_{M,i+1}\subseteq\loc_{M,i}$
	Moreover if $h_m=h_{m+1}$, then $\loc_M=\loc_{M,m}$.
	
	In our case by \cite[Lemma 3.1]{F19} $$\socle(M)=\bigoplus_{i=1}^n k(-d+b_i+3)$$
	so there is no socle in the first half. 
	We can prove similarly the dual condition about the generators, but it is enough to notice that $M$ has the Weak Lefschetz property, so both conditions need to be satisfied. 
	Therefore, we have $$\loc_M =\loc_{\lfloor \frac{d-4}{2}\rfloor-1,M} \cup \loc_{\lfloor \frac{d-4}{2}\rfloor,M}.$$
	
	When $d$ is even, $\loc_M =\loc_{\lfloor \frac{d-4}{2}\rfloor,M}$  because in this case the Hilbert function of $M$ has more than one value equal in the middle.
	
	Moreover, $M$ is Symmetrically Gorenstein by \cite[Proposition 3.9]{F19} and so self-dual up to a twist, then for $d$ odd we have $\loc_{\lfloor \frac{d-4}{2}\rfloor-1,M}= \loc_{\lfloor \frac{d-4}{2}\rfloor,M}$.
	So, in any case $\loc_M =\loc_{\lfloor \frac{d-4}{2}\rfloor,M}.$\end{proof}

Since the non-Lefschetz locus is concentrated in the middle degree, the expected codimension of the non-Lefschetz locus is defined. The main result is prove that such codimension is achieved in the general case.
	\begin{teo}\label{main}
		The non-Lefschetz locus of $M=H_*^1(\Pj^2, \E)$ has the expected codimension for any general rank $2$ vector bundle $\E$ on $\Pj^2$.
		In particular 
		$$\cod(\loc_M)=\begin{cases} 1& \text{ if } \E \text{ unstable;}\\
			1& \text{ if } d=a_1+\dots +a_{n+2}-b_1-\dots -b_n \text{ even;}\\
			2& \text{ if } d \text{ odd and $\E$ stable.}\\
		
		\end{cases}$$
	\end{teo}
\begin{proof}
	We want to prove that $\cod \left(\loc_M\right)=h_{\lfloor \frac{d-4}{2}\rfloor+1}-h_{\lfloor \frac{d-4}{2}\rfloor}+1$. Let us first compute the difference $h_{\lfloor \frac{d}{2}\rfloor-1}-h_{\lfloor \frac{d}{2}\rfloor-2}$.
	
	Shifting the sequence (\ref{1}) 
	$$0\to \mathcal{E} \to \bigoplus_{i=1}^{n+2}\mathcal{O}_{\mathbb{P}^2}(-a_i)\to\bigoplus_{i=1}^{n} \mathcal{O}_{\mathbb{P}^2}(-b_i) \to 0$$
	we can compute the first Chern classes
	$$c_1(\E)=-d \ \ \ \ \ c_1(\E(t))=c_1(\E)+2t=-d+2t.$$
	Then, if $d$ is even, $\En=\E(\frac{d}{2})$ and $c_1(\En)=0$; in the case when $d$ is odd, $\En=\E(\frac{d-1}{2})$ and $c_1(\En)=-1$.

Before starting our proof, first notice that for the non normal vector bundle $\E$ the multiplication map 
$$\times\ell : M_{t-1}=H^1\left(\Pj^2,\E\left(t-1\right)\right)\to M_t=H^1\left(\Pj^2,\E\left(t\right)\right)$$
is the map 
$$\mu_\ell : H^1\left(\Pj^2,\En\left(-\floor*{\frac{d}{2}} + t-1\right)\right)\to H^1\left(\Pj^2,\En\left(-\floor*{\frac{d}{2}} + t\right)\right).$$
Then \cite[Theorem 3.7]{FFP21} implies that for a general linear form $\ell$
\begin{itemize}
	\item if $d$ is even and $\E$ is unstable then $\mu_\ell$ is injective for $ t\leq\frac{d}{2}+ k -1$ and it is surjective for $t\geq\frac{d}{2}-k-1$, where $k>0$ is the instability index of $\En$;
	\item if $d$ is even and $\E$ is semistable then $\mu_\ell$ is injective for $ t\leq\frac{d}{2} -1$ and it is surjective for $t\geq\frac{d}{2}-1$;
	\item if $d$ is odd and $\E$ is unstable then  $\mu_\ell$ is injective for $ t\leq\frac{d-1}{2}+ k$ and it is surjective for $t\geq\frac{d-1}{2}-k-1$, where $k\geq0$ is the instability index of $\En$;
	\item if $d$ is odd and $\E$ is stable then  $\mu_\ell$ is injective for $ t\leq\frac{d-1}{2} -1$ and it is surjective for $t\geq\frac{d-1}{2} -1$.
\end{itemize}
This gives us precise information on the Hilbert function of $M$:
\begin{itemize}
	\item if $d$ is even and $\E$  is unstable then $\cdots h_{\frac{d}{2}-k-3}\leq h_{\frac{d}{2}-k-2}=\dots=h_{\frac{d}{2}+k-1}\geq h_{\frac{d}{2}+ k}\cdots$. In particular $h_{\frac{d}{2}-3}=h_{\frac{d}{2}-2}=h_{\frac{d}{2}-1}=h_{\frac{d}{2}}$ since  $k>0$ ;
	\item if $d$ is even and $\E$ is semistable then $\cdots h_{\frac{d}{2}-3}\leq h_{\frac{d}{2}-2}=h_{\frac{d}{2}-1}\geq h_{\frac{d}{2}}\cdots$;
	\item if $d$ is odd and $\E$ is unstable then $\cdots h_{\frac{d-1}{2}-k-3} \leq h_{\frac{d-1}{2}-k-2}=\dots=h_{\frac{d-1}{2}+k}\geq h_{\frac{d-1}{2}+ k +1}\cdots$. In particular $h_{\frac{d-1}{2}-2}=h_{\frac{d-1}{2}-1}=h_{\frac{d-1}{2}}$ since  $k\geq0$ ;
	\item if $d$ is odd and $\E$ is stable then $\cdots h_{\frac{d-1}{2}-2}\leq h_{\frac{d-1}{2}-1}\geq h_{\frac{d-1}{2}}\cdots$.
\end{itemize}
\emph{Case 1}. $\E$ either unstable or $d$ even. 

In the first 3 cases above the Hilbert function has more than one value equal in the middle, hence 
$$\exc(\loc_M)=h_{\lfloor \frac{d-4}{2}\rfloor+1}-h_{\lfloor \frac{d-4}{2}\rfloor}+1=h_{\lfloor \frac{d}{2}\rfloor-1}-h_{\lfloor \frac{d}{2}\rfloor-2}+1=1.$$

Since $M$ has the Weak Lefschetz property by \cite[Theorem 3.7]{FFP21}, we can conclude that if either $\E$ unstable or for $d$ even, then the expected codimension is always achieved:
$$1\leq \cod (\loc_M)\leq \exc \left(\loc_{M}\right)=1.$$
So $\loc_M$ is a hypersurface in $\Pj^2$.

\emph{Case 2}.
Let us now assume $\E$ is a general stable vector bundle and $d$ odd, i.e. $c_1(\En)=-1$.

 In this case using the fact that the non-Lefschetz locus is the set of the jumping lines (Proposition \ref{jum}), by \cite[Corollary 10.7.1]{16} this set is finite, so the non-Lefschetz locus has codimension 2.
Let us compute explicitly the expected codimension: 
\begin{align*}
	\exc \left(\loc_M\right)=&h_{\lfloor \frac{d}{2}\rfloor-1}-h_{\lfloor \frac{d}{2}\rfloor-2}+1=h_{\frac{d-3}{2}}-h_{ \frac{d-5}{2}}+1\\
	=&\ h^1\left(\Pj^2,\E\left(\frac{d-3}{2}\right)\right)-h^1\left(\Pj^2,\E\left(\frac{d-5}{2}\right)\right)\\
	=&\ -\chi\left(\frac{d-3}{2}\right)+h^0\left(\Pj^2,\E\left(\frac{d-3}{2}\right)\right)+h^2\left(\Pj^2,\E\left(\frac{d-3}{2}\right)\right)\\&+\chi\left(\frac{d-5}{2}\right)-h^0\left(\Pj^2,\E\left(\frac{d-5}{2}\right)\right)-h^2\left(\Pj^2,\E\left(\frac{d-5}{2}\right)\right)+1\\
	=&\ \chi\left(\frac{d-5}{2}\right)-\chi\left(\frac{d-3}{2}\right)+1\\&-h^0\left(\Pj^2,\E\left(\frac{d-1}{2}\right)\right)+2h^0\left(\Pj^2,\E\left(\frac{d-3}{2}\right)\right)-h^0\left(\Pj^2,\E\left(\frac{d-5}{2}\right)\right)\\
	=&\ \chi\left(\frac{d-5}{2}\right)-\chi\left(\frac{d-3}{2}\right)+1
\end{align*}
where we have used the facts that $\En^\vee=\En(-1)$, and by Serre Duality, $h^0(\Pj^2,\E(t))=0$ for any $t\leq -c_1(\E)/2=d/2$. In particular $$h^0\left(\Pj^2,\E\left(\frac{d-1}{2}\right)\right)=h^0\left(\Pj^2,\E\left(\frac{d-3}{2}\right)\right)=h^0\left(\Pj^2,\E\left(\frac{d-5}{2}\right)\right)=0$$ as we saw in the proof of Proposition \ref{b}.

To compute the Euler characteristic of $\E(t)$ we shift the  sequence (\ref{1})
$$0\to \mathcal{E}(t) \to \bigoplus_{i=1}^{n+2}\mathcal{O}_{\mathbb{P}^2}(t-a_i)\to\bigoplus_{i=1}^{n} \mathcal{O}_{\mathbb{P}^2}(t-b_i) \to 0,$$
and we use additivity
\begin{align*}
	\chi(\E(t))&=\ \chi\left(\bigoplus_{i=1}^{n+2}\mathcal{O}_{\mathbb{P}^2}(t-a_i)\right) -\chi\left(\bigoplus_{i=1}^{n} \mathcal{O}_{\mathbb{P}^2}(t-b_i)\right)\\
	&=\ \sum_{i=1}^{n+2}\chi\left(\mathcal{O}_{\mathbb{P}^2}(t-a_i)\right) -\sum_{i=1}^{n} \chi\left(\mathcal{O}_{\mathbb{P}^2}(t-b_i)\right)\\
	&=\ \sum_{i=1}^{n+2}\frac{(t-a_i+2)(t-a_i+1)}{2} -\sum_{i=1}^{n} \frac{(t-b_i+2)(t-b_i+1)}{2}\\
	&=\ t^2+3t+2-dt-\frac{3}{2}d +\frac{1}{2}\left(\sum_{i=1}^{n+2}a_i^2-\sum_{i=1}^{n}b_i^2\right).
\end{align*}
Finally, 
\begin{align*}
	\exc \left(\loc_M\right)=
	&\ \chi\left(\frac{d-5}{2}\right)-\chi\left(\frac{d-3}{2}\right)+1\\
	=&\ \left(\frac{d-5}{2}\right)^2+3\frac{d-5}{2}-d\frac{d-5}{2}-\left(\frac{d-3}{2}\right)^2-3\frac{d-5}{2}+d\frac{d-5}{2}+1=2.
\end{align*}
This proves that  $\exc \left(\loc_M\right)=\cod(\loc_M)=2$, as we wanted.\end{proof}
\begin{rmk}
	By Remark \ref{deg}, we can conclude that for a general rank $2$ vector bundle $\E$ on $\Pj^2$ the non-Lefschetz locus of $M=H_*^1(\Pj^2, \E)$ is
	\begin{itemize}
		\item a hypersurface of degree $h_{\lfloor \frac{d-4}{2}\rfloor}$ if either $\E$ unstable or $d$ is even (here the hypothesis of generality is not needed);
		\item a finite set of $\binom{h_{\frac{d-1}{2}}}{h_{ \frac{d-3}{2}}-1}=\binom{c_2(\E(\frac{d-1}{2}))}{2}$ points if
		 $\E$ is stable and $d$ is odd. The last equality comes from the description of the set of the jumping lines in   \cite[Corollary 10.7.1]{16} since in this case $\En=\E(\frac{d-1}{2})$.
	\end{itemize}
\end{rmk}

\bibliographystyle{amsalpha}
\bibliography{non-Lefschetz_locus_vector_bundle_P2}

\providecommand{\bysame}{\leavevmode\hbox to3em{\hrulefill}\thinspace}
\providecommand{\MR}{\relax\ifhmode\unskip\space\fi MR }
\providecommand{\MRhref}[2]{%
  \href{http://www.ams.org/mathscinet-getitem?mr=#1}{#2}
}
\providecommand{\href}[2]{#2}
\begin{thebibliography}{BMMRN18}

\bibitem[BK07]{3}
Holger Brenner and Almar Kaid, \emph{Syzygy bundles on $\mathbb{P}^2$ and the
  weak lefschetz property}, Illinois J. Math (2007), 1299--1308.

\bibitem[BMMRN18]{main}
Mats Boij, Juan Migliore, Rosa~M. Mir\'o-Roig, and Uwe Nagel, \emph{The
  non-lefschetz locus}, Journal of Algebra \textbf{505} (2018), 288 -- 320.

\bibitem[BS92]{1}
Guntram Bohnhorst and Heinz Spindler, \emph{The stability of certain vector
  bundles on $\mathbb{P}^n$}, Complex Algebraic Varieties (Berlin, Heidelberg)
  (Klaus Hulek, Thomas Peternell, Michael Schneider, and Frank-Olaf Schreyer,
  eds.), Springer Berlin Heidelberg, 1992, pp.~39--50.

\bibitem[FFP21]{FFP21}
Gioia Failla, Zachary Flores, and Chris Peterson, \emph{On the weak {L}efschetz
  property for vector bundles on $\mathbb{P}^2$}, Journal of Algebra
  \textbf{568} (2021), 22--34.

\bibitem[{Flo}19]{F19}
Zachary {Flores}, \emph{{Symmetry, Unimodality, and Lefschetz Properties for
  Graded Modules}}, arXiv e-prints (2019), arXiv:1908.03648.

\bibitem[Har77]{Hartshorne}
R.~Hartshorne, \emph{Algebraic geometry}, Graduate Texts in Mathematics,
  Springer, 1977.

\bibitem[HMNW03]{13}
T.~Harima, J.~Migliore, U.~Nagel, and J.~Watanabe, \emph{The weak and strong
  {L}efschetz properties for {A}rtinian k-algebras}, Journal of Algebra
  \textbf{262} (2003).

\bibitem[Hul79]{16}
Klaus Hulek, \emph{{Stable rank-2 vector bundles on $\mathbb{P}^2$ with $c_1$
  odd}}, Mathematische Annalen \textbf{242} (1979), 241--266.

\bibitem[Mar21]{io}
Emanuela Marangone, \emph{{Some notes and corrections of the paper ``The
  non-Lefschetz locus''}}.

\bibitem[Mig86]{MIGLIORE}
Juan Migliore, \emph{Geometric invariants for liaison of space curves}, Journal
  of Algebra \textbf{99} (1986), no.~2, 548--572.

\bibitem[OSS88]{OSS}
C.~Okonek, M.~Schneider, and H.~Spindler, \emph{Vector bundles on complex
  projective spaces}, Birkhäuser, 1988.

\bibitem[RRR91]{26}
Les Reid, Leslie~G. Roberts, and Moshe Roitman, \emph{On complete intersections
  and their hilbert functions}, Canadian Mathematical Bulletin \textbf{34}
  (1991), no.~4, 525–535.

\bibitem[Sta80]{28}
Richard~P. Stanley, \emph{Weyl groups, the hard lefschetz theorem, and the
  sperner property}, SIAM Journal on Algebraic Discrete Methods \textbf{1}
  (1980), no.~2, 168--184.

\bibitem[Wat87]{29}
Junzo Watanabe, \emph{The dilworth number of artinian rings and finite posets
  with rank function}, Mathematical Society of Japan, 1987, pp.~303--312.

\end{thebibliography}

\end{document}